\documentclass{amsart}

\usepackage{amsmath,amsthm,amsfonts, amssymb,amscd}
\usepackage[all]{xy}

\newtheorem{theorem}{Theorem}[section]
\newtheorem{lemma}[theorem]{Lemma}%[section]
\newtheorem{remark}[theorem]{Remark}%[section]
\newtheorem{corollary}[theorem]{Corollary}%[section]
\newtheorem{proposition}[theorem]{Proposition}

\newcommand{\minusre}{\hspace{0.3em}\raisebox{0.3ex}{\sl \tiny /}\hspace{0.3em}}
\newcommand{\minusli}{\hspace{0.3em}\raisebox{0.3ex}{\sl \tiny $\setminus $}\hspace{0.3em}}

\newcommand{\Ker}{\mbox{\rm Ker}}

\begin{document}
\title{Decompositions  of  Measures on Pseudo Effect Algebras }
\author{Anatolij Dvure\v censkij}
\date{}%Jan. 5, 2010
\maketitle

\begin{center}
\footnote{


The  author thanks  for the support by Center of Excellence SAS
-~Quantum Technologies~-,  ERDF OP R\&D Projects CE QUTE ITMS
26240120009 and meta-QUTE ITMS 26240120022, the grant VEGA No.
2/0032/09 SAV. }
\small{Mathematical Institute,  Slovak Academy of Sciences\\
\v Stef\'anikova 49, SK-814 73 Bratislava, Slovakia\\
E-mail: {\tt dvurecen@mat.savba.sk}, }
\end{center}

\begin{abstract}
Recently in \cite{Dvu3} it was shown that if a pseudo effect algebra
satisfies a kind of the Riesz Decomposition Property ((RDP) for
short), then its state space is either empty or a nonempty simplex.
This will allow us to prove a Yosida--Hewitt type and a Lebesgue
type decomposition for measures on pseudo effect algebra with (RDP).
The simplex structure of the state space will entail not only the
existence of such a decomposition but also its uniqueness.
\end{abstract}

\section{Introduction} %1

The classical decomposition theorems of   or Yosida--Hewitt
\cite{YoHe} initiated in the last two decades interest of authors
studying finitely additive measures on quantum structures like
orthomodular lattices or posets to study an interesting problem of
decomposition measures. There appeared a whole series of papers
studying Lebesgue and Yosida-Hewitt type decompositions, see e.g.
\cite{DDP, DeMo, Rut1, Rut2}. They exhibited at least the existence
of such a decomposition.  To prove even the uniqueness of
decompositions, some sufficient conditions are presented in
\cite{Rut2}.

Quantum structures were inspired by the research of the mathematical
foundations of quantum structures. An analogue of a probability
measure is a {\it state}. One of the most important examples of
orthomodular lattices or of the Hilbert space quantum mechanics is
the system $\mathcal L(H)$ of all closed subspaces of a Hilbert
space (real, complex or quaternionic) $H.$  The Gleason Theorem, see
e.g. \cite{Dvu0}, says that every $\sigma$-additive state on
$\mathcal L(H)$ is uniquely expressible via a Hermitian trace
operator on $H$ if $3\le \dim H\le \aleph_0.$ The Aarnes Theorem,
\cite[Thm 3.2.28]{Dvu0} says that every (finitely additive) state on
$\mathcal L(H)$ is a unique convex combination of two states, $s_1$
and $s_2,$ where $s_1$ is a completely additive state and $s_2$ is a
finitely additive state that vanishes on each finite-dimensional
subspace of $H.$

In the Nineties, Foulis and Bennett \cite{FoBe} introduced effect
algebras that are partial structures with a partially defined
operation $+$ that models the join of mutually excluding events.
They generalize orthomodular lattices and posets, orthoalgebras, and
the basic example important for so-called POV-measures of quantum
mechanics is the system, $\mathcal E(H),$ of all Hermitian operators
of a Hilbert space $H$ that are between the zero operator and the
identity one.

These commutative structures were extended in \cite{DvVe1, DvVe2} to
so-called {\it pseudo effect algebras} where the partial addition,
$+,$ is not more assumed to be commutative. In many important
examples they are  intervals in po-groups (= partially ordered
groups). E.g. $\mathcal E(H)$ is an interval in the po-group
$\mathcal B(H)$ of all Hermitian operators on $H.$

A sufficient condition for a pseudo effect algebra to be an interval
of a unital po-group is a variant of the Riesz Decomposition
Property ((RDP) in abbreviation), see \cite{Dvu2, DvVe2}. It is a
weaker form of the distributivity that allows to do a joint
refinement of two decompositions of the unit element $1.$ For
example, (RDP) on an orthomodular poset entails that it has to be a
Boolean algebra, and therefore, (RDP) fails to hold on $\mathcal
L(H)$ or on $\mathcal E(H).$

We recall that every effect algebra with (RDP) has at least one
state, however the state space of a pseudo effect algebra with a
stronger type of (RDP) can be empty, \cite{Dvu1}.

Recently in \cite[Thm 5.1]{Dvu3}, it was shown that the state space
of every pseudo effect algebra with (RDP) is a simplex, more
precisely a Choquet simplex. The simplex is a special type of a
convex set that generalizes the classical one in $\mathbb R^n.$ For
a comprehensive source on simplices see \cite{Alf}. We note that the
state space of $\mathcal E(H)$ is not a simplex, however the state
space of any commutative C$^*$-algebra is a simplex, see e.g.
\cite[Thm 4.4, p. 7]{AlSc}. The simplex structure of the state
spaces allows also to represent uniquely  states as an integral
\cite{Dvu4} through a regular Borel probability measure.

The simplex structure of the state space of a pseudo effect algebra
$E$ satisfying  (RDP) entails that the  space of all Jordan measures
on $E$ is an Abelian Dedekind complete $\ell$-group, \cite[Thm 3.5,
Thm 3.6]{Dvu3}. This new fact is our basic tool to present   a
Yosida--Hewitt type and a Lebesgue type of decompositions of
finitely additive measures on $E.$ The property (RDP) as we show is
a sufficient condition to prove not only the existence but also the
uniqueness of such a decomposition that is a main goal of the
present paper.

The paper is organized as follows.

Section 2 is an introduction to the theory of pseudo effect algebras
gathering the necessary latest results.  Section 3 describes the
faces of the state space of pseudo effect algebras and it gives a
general result on a decomposition and it will be applied in
Section 4 to present the main body of the paper~- the Yosida--Hewitt
type of decomposition and the Lebesque type decomposition.

\section{Elements of Pseudo Effect Algebras}%2

Pseudo effect algebras were introduced in \cite{DvVe1, DvVe2}. We
say that a {\it pseudo effect algebra} is a partial algebra  $(E; +,
0, 1)$, where $+$ is a partial binary operation and $0$ and $1$ are
constants, such that for all $a, b, c \in E$, the following holds

\begin{enumerate}
\item[(i)] $a+b$ and $(a+b)+c$ exist if and only if $b+c$ and
$a+(b+c)$ exist, and in this case $(a+b)+c = a+(b+c)$;

\item[(ii)]
  there is exactly one $d \in E$ and
exactly one $e \in E$ such that $a+d = e+a = 1$;

\item[(iii)]
 if $a+b$ exists, there are elements $d, e
\in E$ such that $a+b = d+a = b+e$;

\item[(iv)] if $1+a$ or $a+1$ exists, then $a = 0$.
\end{enumerate}

If we define $a \le b$ if and only if there exists an element $c\in
E$ such that $a+c =b,$ then $\le$ is a partial ordering on $E$ such
that $0 \le a \le 1$ for any $a \in E.$ It is possible to show that
$a \le b$ if and only if $b = a+c = d+a$ for some $c,d \in E$. We
write $c = a \minusre b$ and $d = b \minusli a.$ Then

$$ (b \minusli a) + a = a + (a \minusre b) = b,
$$
and we write $a^- = 1 \minusli a$ and $a^\sim = a\minusre 1$ for any
$a \in E.$

For basic properties of pseudo effect algebras see \cite{DvVe1,
DvVe2}. We recall that if $+$ is commutative, $E$ is said to be an
{\it effect algebra}; for a guide overview on effect
algebras we recommend e.g. \cite{DvPu}.

For example, let $(G,u)$ be a unital po-group (= partially ordered
group) with strong unit $u$ that is not necessarily Abelian.  We
recall that a po-group (= partially ordered group) is a group with a
partial ordering $\le$ such that if $a\le b,$ then $x+a+y\le x+b+y$
for all $x,y \in G;$ and an element $u \in G^+:=\{g\in G: g\ge 0\}$
is said to be a {\it strong unit} if given $G = \bigcup_n[-nu,nu].$

Then for $(G,u)$ we define $\Gamma(G,u)=[0,u]$ and we endow it with
$+$ that is the restriction of the group addition to the set of all
those $(x,y)\in \Gamma(G,u)\times \Gamma(G,u)$ that $x\le u-y.$ Then
$(\Gamma(G,u);+,0,u)$ is a pseudo effect algebra with possible two
negations: $a^-=u-a$ and $a^\sim = -a+u.$ Any pseudo effect algebra
of the form $\Gamma(G,u)$ is said to be an {\it interval pseudo
effect algebra}. In \cite{DvVe1, DvVe2}, we have some sufficient
conditions posed to a pseudo effect algebra to be an interval.  They
are analogues  of the Riesz Decomposition Properties. Roughly
speaking they are a weaker form of distributivity that allows a
joint refinement of two partitions of $1.$ This is a reason why they
fail to hold for $\mathcal L(H)$ and $\mathcal E(H).$

Now we introduce according to \cite{DvVe1}  the following types of
the Riesz Decomposition properties for pseudo effect algebras that
in the case of effect algebras may coincide, but not for pseudo
effect algebras, in general.

\begin{enumerate}
 \item[(a)] For $a, b \in E$, we write $a \ \mbox{\bf com}\ b$ to mean that
for all $a_1 \leq a$ and $b_1 \leq b$, $a_1$ and $b_1$ commute.

\item[(b)] We say that $E$ fulfils the {\it Riesz
Interpolation Property}, (RIP) for short, if for any $a_1, a_2, b_1,
b_2 \in E$ such that $a_1, a_2 \,\leq\, b_1, b_2$ there is a $c \in
E$ such that $a_1, a_2 \,\leq\, c \,\leq\, b_1, b_2$.

 \item[(c)] We say that $E$ fulfils the {\it weak Riesz
Decomposition Property}, (RDP$_0$) for short, if for any $a, b_1,
b_2 \in E$ such that $a \leq b_1+b_2$ there are $d_1, d_2 \in E$
such that $d_1 \leq b_1$, $\;d_2 \leq b_2$ and $a = d_1+d_2$.

 \item[(d)] We say that $E$ fulfils the {\it Riesz
Decomposition Property}, (RDP) for short, if for any $a_1, a_2, b_1,
b_2 \in E$ such that $a_1+a_2 = b_1+b_2$ there are $d_1, d_2, d_3,
d_4 \in E$ such that $d_1+d_2 = a_1$, $\,d_3+d_4 = a_2$, $\,d_1+d_3
= b_1$, $\,d_2+d_4 = b_2$.

 \item[(e)] We say that $E$ fulfils the {\it
commutational Riesz Decomposition Property}, (RDP$_1$) for short, if
for any $a_1, a_2, b_1, b_2 \in E$ such that $a_1+a_2 = b_1+b_2$
there are $d_1, d_2, d_3, d_4 \in E$ such that (i) $d_1+d_2 = a_1$,
$\,d_3+d_4 = a_2$, $\,d_1+d_3 = b_1$, $\,d_2+d_4 = b_2$, and (ii)
$d_2 \ \mbox{\bf com}\ d_3$.

 \item[(f)] We say that $E$ fulfils the {\it strong
Riesz Decomposition Property}, (RDP$_2$) for short, if for any $a_1,
a_2, b_1, b_2 \in E$ such that $a_1+a_2 = b_1+b_2$ there are $d_1,
d_2, d_3, d_4 \in E$ such that (i) $d_1+d_2 = a_1$, $\,d_3+d_4 =
a_2$, $\,d_1+d_3 = b_1$, $\,d_2+d_4 = b_2$, and (ii) $d_2 \wedge d_3
= 0$.
\end{enumerate}

We have the implications

\centerline{ (RDP$_2$) $\Rightarrow$ (RDP$_1$) $\Rightarrow$ (RDP)
$\Rightarrow$\ (RDP$_0$) $\Rightarrow$ (RIP). }

The converse of any of these implications does not hold. For
commutative effect algebras we have

\centerline{ (RDP$_2$) $\Rightarrow$ (RDP$_1$) $\Leftrightarrow$
(RDP) $\Leftrightarrow$\ (RDP$_0$) $\Rightarrow$ (RIP). }

If in the above definitions of (RDP)'s we change $E$ to $G^+,$ we
have po-groups with the corresponding forms of the Riesz
Decomposition Properties. According to \cite[Thm 5.7]{DvVe2}, every
pseudo effect algebra with (RDP)$_1$ is an interval in some unital
po-group with (RDP)$_1,$ so is any effect algebra with (RDP) in an
Abelian po-group with (RDP), see see \cite{Rav} or \cite[Thm
1.7.17]{DvPu}. Any effect algebra with (RDP) that is a lattice, or
equivalently, with (RDP)$_2$ is an MV-algebra, for a definition see
e.g. \cite{DvPu}. Any pseudo effect algebra with (RDP)$_2$ is a
so-called pseudo MV-algebra, see \cite{GeIo}, and the group $G$ is
an $\ell$-group (= lattice ordered group).

A {\it signed measure} on a pseudo effect algebra $E$ is any mapping
$m:E \to  \mathbb R$ such that $m(a+b)=m(a)+m(b)$ provided $a+b$ is
defined in $E.$ We have $s(0)=0$ and $s(a^-)=s(a^\sim).$ If a signed
measure $m$ is positive, i.e., $m(a)\ge 0$ for each $a \in E,$ we
call it a {\it measure}, and any normalized measure, i.e. a measure
$s$ such that $s(1)=1,$ is said to be a {\it state}. For any measure
$m,$ we have $m(a)\le m(b)$ whenever $a\le b.$ If $m_1$ and $m_2$
are two measures on $E$, then the signed measure $m=m_1-m_2$ is said
to be a {\it Jordan signed measure.} We denote by $\mathcal M(E),$
$\mathcal M^+(E),$ $\mathcal S(E),$ and $\mathcal J(E)$ the sets of
all signed measures, or measures, or states or Jordan signed
measures on $E.$ We recall that it can happen that $\mathcal
M(E)=\{0\}=\mathcal J(E).$

The set $\mathcal S(E)$ is convex, i.e., any convex combination
$s=\lambda s_1 +(1-\lambda)s_2,$ $\lambda \in [0,1],$ of two states
$s_1$ and $s_2$ and $\lambda \in [0,1]$ is a state. If $s$ cannot be
expressed by a convex combination of two distinct states, it is
called an {\it extremal state}. Let $\partial_e \mathcal S(E)$
denote the set of all extremal states. On $\mathcal M(E)$ we
introduce a {\it weak topology}: We say that a net of measures
$\{m_\alpha\}$ {\it converges weakly} to a measure $m$ if
$\lim_\alpha m_\alpha(a)=m(a).$ Then $\mathcal S(E)$ is a convex
compact Hausdorff space, and due to the Krein--Mil'man Theorem, see
\cite[Thm 5.17]{Goo}, every state on $E$ is a weak limit of a net of
convex combinations of extremal states.

If $E$ is a pseudo effect algebra with (RDP), then $\mathcal S(E)$
is either empty or a non-void simplex, \cite[Thm 5.1]{Dvu3}, for
definition of a simplex and its basic properties, see \cite[Chap
10]{Goo}.

For two signed measures $m_1$ and $m_2,$ we write $m_1\le^+ m_2$ if
$m_1(a)\le m_2(a)$ for each $a\in E.$

The following important statement was proved in \cite[Thm 3.5, Thm
3.6]{Dvu3}:

\begin{theorem}\label{th:2.1}
Let $E$ be a pseudo effect algebra with {\rm (RDP)}. Then $\mathcal
J(E)$ is an Abelian Dedekind complete $\ell$-group such that if
$\{m_i\}_{i\in I}$ is a nonempty system of $\mathcal J(E)$ that is
bounded above, and if $d(x)=\bigvee_i m_i(x)$ for all $x \in E,$
then
$$ \left(\bigvee_i m_i\right)(x) = \bigvee\{d(x_1)+\cdots + d(x_n):
x= x_1+\cdots + x_n, \ x_1,\ldots, x_n \in E\}
$$
for all $x \in E.$

And if $e(x)=\bigwedge_i f_i(x)$ for all $x
\in E,$ then
$$ \left(\bigwedge_i m_i\right)(x) = \bigwedge\{e(x_1)+\cdots + e(x_n):
x= x_1+\cdots + x_n, \ x_1,\ldots, x_n \in E\}
$$
for all $x \in E.$

Given $m_1,\ldots, m_n \in \mathcal J(E)$,

\begin{eqnarray*}
\left( \bigvee_{i=1}^n m_i\right)(x) = \sup\{m_1(x_1)+\cdots
+m_n(x_n):
x = x_1+\cdots +x_n,\ x_1,\ldots, x_n \in E\},\\
\left( \bigwedge_{i=1}^n m_i \right)(x) = \inf\{m_1(x_1)+\cdots
+m_n(x_n):
x = x_1+\cdots +x_n,\ x_1,\ldots, x_n \in E\},\\
\end{eqnarray*}
for all $x \in E.$
\end{theorem}

\section{Faces of the State Space}%3

The present section describes the faces of the state space of a
pseudo effect algebra with (RDP). Since our state space is a
simplex, we know that if $F$ is a closed face, then every simplex is
a direct convex sum of $F$ and its complementary face, \cite[Thm
11.28]{Goo}.  However, not every face is closed, we present a
general decomposition, Theorem \ref{th:3.4}, where a weaker form of
the closedness, $\vee$-closedness, allows to obtain a unique
decomposition of measures.  This result will apply in the next
section  to obtain Yosida--Hewitt and Lebesgue types of
decomposition.

A {\it face} of a convex set $K$ is a convex subset $F$ of $K$ such
that if $x= \lambda x_1+(1-\lambda)x_2 \in F$ for $\lambda \in
(0,1),$ then also $x_1,x_2 \in F.$ We note that if $x \in K,$ then
$\{x\}$ is a face iff $x \in
\partial_e K.$

For any $X \subseteq K,$ there is the face generated by $X.$ Due to
\cite[Prop 5.7]{Goo}, the face $F$ generated by $X$ is the set of
those points $x \in K$ for which there exists a positive convex
combination $\lambda x + (1-\lambda) y= z$ with $y \in K$ and $z$
belongs to the convex hull of $X.$

If $\mathcal S(E)\ne \emptyset,$ then a state $s\in \mathcal S(E)$
belongs to the face generated by $X$ if and only if $s\le ^+ \alpha
t$ for some positive constant $\alpha$ and some state $t$ in the
convex hull of $X.$ In particular, the face of $\mathcal S(E)$
generated by a state $s$ is the set of states $s' \in \mathcal S(E)$
such that $s'\le^+ \alpha s$ for some real number $\alpha> 0.$

Let $s$ be a state on a pseudo effect algebra $E.$ The {\it kernel}
of $s$ is the set

$$\Ker(s):=\{x \in E: s(x)=0\}.$$
Then $\Ker(s)$ is a normal ideal of $E.$  We note that a subset $I$
of a pseudo effect algebra is an {\it ideal} if (i) $0\in I,$ (ii)
$a,b \in I$ and $a+b\in E$ imply $a+b \in I,$ and (iii) $a\in E,$ $b
\in I$ and $a\le b$ entail $a\in I.$ An ideal $I$ is {\it normal},
if $a+I:=\{a+b \in E: b \in I\}= I+a :=\{b+a\in E: b \in I\}$ for
any $a \in E.$

\begin{proposition}\label{pr:3.1}  Let $E$ be a pseudo effect
algebra and let $X$ be a subset of $E.$  Then the set

$$F=\{s \in \mathcal S(E): X \subseteq \Ker(s)\}$$
is a closed face of $\mathcal S(E).$
\end{proposition}

\begin{proof}  If  $F=\emptyset,$ then $F$ is trivially a closed
face.  Assume $F\ne \emptyset.$  If $s = \lambda s_1
+(1-\lambda)s_2$ for $\lambda \in  [0,1]$ and for two states $s_1$
and $s_2$, then $\Ker(s)\subseteq \Ker(s_1)\cap \Ker(s_2)$ which
proves that $F$ is a convex set. If $\{s_\alpha\}$ is a net of
states from $F$ that converges weakly to a state $s$ on $E$, then
for each $x\in X$, $0=\lim_\alpha s_\alpha (x)=s(x)$ so that $F$ is
closed. If now $s = \lambda s_1 +(1-\lambda)s_2 \in F$ for $\lambda
\in (0,1)$ and $s_1,s_2 \in F,$ then for each $x\in X$ we have
$0=s(x) = \lambda s_1(x) +(1-\lambda)s_2(x)$ so that
$s_1(x)=s_2(x)=0$ and therefore, $s_1,s_2 \in F.$
\end{proof}

Let $F$ be a face of a simplex $K$, and let $F'$ be the union of
those faces of $K$ that are disjoint from $F.$  According to
\cite[Prop 10.12]{Goo}, $F'$ is a face of $K$ and it is the largest
face of $K$ that is disjoint from $F.$ If a point $x \in K$ can be
expressed as convex combinations
$$x = \alpha_1 x_1 +\alpha_2 x_2 = \beta_1 y_1 + \beta_2 y_2
$$
with $x_1,y_1 \in F$ and $x_2,y_2\in F',$ then $\alpha_i =\beta_i$
for $i=1,2$ and $x_i=y_i$ for those $i$ such that $\alpha_i >0.$ The
face  $F'$ is said to be a {\it complementary face} of $F$ in $K.$

It is possible to show that if $F_1\subseteq F_2$ are two faces of
$\mathcal S(E),$ then $F_2'\subseteq F_1'$ and for a set $\{F_i\}$
of faces, we have $F=\bigcap_i F_i$  is a face, and

$$ \left(\bigcap_i F_i\right)' = \bigcup_i F_i'\eqno(3.0).$$

\begin{proposition}\label{pr:3.2}  Let $F$ be a  face of the
state space of a pseudo effect algebra $E$ with {\rm (RDP)} and let
$\mathcal S(E)\ne \emptyset.$ Then its complementary face $F'$
consists of all states $s' \in \mathcal S(E)$ such that $s'\wedge s
= 0$ for every state $s \in F.$

Equivalently, $F'$ consists of all states $s' \in \mathcal S(E)$
such that if $s \in F$ is such that $\alpha s \le^+s'$ for some
constant $\alpha \ge 0,$ then $\alpha = 0.$
\end{proposition}

\begin{proof}

We know that $\mathcal S(E)$ is a simplex. According to \cite[Prop
10.12]{Goo},  $F'$ is the union of all faces of $\mathcal S(E)$ that
are disjoint from $F.$ If $s' \in F'$ and $s \in F,$ and if $s'$ and
$s$ belong to mutually disjoint faces, then $s'\wedge s=0,$ see
\cite[Prop 4.2]{Dvu3}. Conversely, let $s'$ be a state on $E$ such that $s'\wedge
s=0$ for each $s \in F.$ Then by \cite[Prop 4.2]{Dvu3}, $s'$ and $s$ belong to
mutually disjoint faces. Let $F(s')$ and $F(s)$ be the faces
generated by $s'$ and $s$.  Therefore, $F(s')\cap F(s)=\emptyset.$
But $F = \bigcup_{s\in F}F(s),$ so that $F(s')\cap F = \bigcup_{s\in
F}F(s')\cap F(s)=\emptyset$ that gives $s'\in F'$

Now let $s' \in F'$ and let $s\in F$ be such that $\alpha s\le s'$
for some $\alpha \ge 0.$  If $\alpha >0,$ then $s\le s'/\alpha$ that
$s$ belongs to the face generated by $s'$ that is impossible, whence
$\alpha =0.$

Assume that $F_0$ is the set of all those states $s'$ on $E$ such
that if $\alpha s\le^+ s'$ for $s\in F$ and $\alpha \ge 0,$ then
$\alpha =0.$ We have seen that $F' \subseteq F_0.$ Let $s'\in F_0$
and $s\in F$  If $s' \wedge s>^+0,$ then  $s$ belongs to the face
generated by $s'$, hence, $s\le^+ ts'$ for some $t>0.$ This gives
$s/t \le^+ s$ so that $1/t=0$ that is absurd. Hence, $F_0 \subseteq
F'.$
\end{proof}

Let $K_1,\ldots,K_n$ be convex subsets of $K.$ We say that $K$ is
the {\it direct convex sum} of $K_1,\ldots,K_n$ if (i) $K$ equals
the convex hull of $\bigcup_{i=1}^n K_i$ and every element $x\in K$
can be uniquely expressed as a convex combination of some elements
$x_i\in K_i,$ $i=1,\ldots,n.$  That is, if

$$ \alpha_1x_1+\cdots+\alpha_n x_n = \beta_1 y_1+\cdots+ \beta_n
y_n$$ are convex combinations of $x_i,y_i \in K_i,$ for
$i=1,\ldots,n,$ then $\alpha_i =\beta_i $ for all $i$'s and
$x_i=y_j$ for those $i$ such that $\alpha_i >0.$

\begin{theorem}\label{th:3.3}  Let $E$  be a pseudo effect algebra
with {\rm (RDP)} and let $\mathcal S(E)\ne \emptyset.$ Let $I$ be a
normal ideal of $E,$ and set

$$ F =\{s\in \mathcal S(E): I\subseteq \Ker(s)\}.
$$
Then $F$ is a closed face of $\mathcal S(E),$ and $\mathcal S(E)$ is
the direct convex sum of $F$ and its complementary face.
\end{theorem}

\begin{proof} By Proposition \ref{pr:3.1}, $F$ is a closed face of
$\mathcal S(E).$ If $F$ is empty, then $F'=\mathcal S(E).$

Now assume $F$ is non-void. Let $\mathcal J(E)$ be the set of all
Jordan signed measures on $E.$ By Theorem \ref{th:2.1}, $\mathcal
J(E)$ is an Abelian Dedekind complete $\ell$-group, and $\mathcal
S(E)$ is a base  for the positive cone $\mathcal J(E)^+=\{\alpha s:
\alpha \ge 0, s \in \mathcal S(E)\}.$ We show that the set

$$C =\{\alpha s: \alpha \ge 0,\ s\in \mathcal S(E),\ s\in F\}
$$
contains  the supremum of any subset of $C$ which is bounded above
in $\mathcal J(E).$  Given a nonempty system of positive measures
$\{m_i\}_i$ from $C$ that is bounded in $\mathcal J(E),$ let $m =
\bigvee_i m_i.$ According to Theorem \ref{th:2.1},
$$m(x) =\sup\{d(x_1)+\cdots+d(x_n): x= x_1+\cdots +x_n\},\quad x \in E,$$
where $d(x)=\sup_i m_i(x),$ $x \in E.$

If $x \in I$, then $x=x_1+\cdots+x_n$ entails $x_1,\ldots, x_n \in
I$ so that $d(x_j)= \sup_i m_i(x_j) =0$ for each $j=1,\ldots, n.$
Hence, $m(x)=0$ and $m\in C.$

According to \cite[Prop 10.14]{Goo}, $\mathcal S(E)$ is the direct
convex sum of $F$ and its complementary face.
\end{proof}

Let $F$ be a face of $\mathcal S(E),$ where $E$ is a pseudo effect
algebra with (RDP). We set
$$V(F):=\{\alpha s: \alpha\ge 0, s\in F\}.
$$
It is a cone that is a subcone of $\mathcal M^+(S),$ i.e., if (i)
$m_1,m_2 \in V(F),$ then $m_1+m_2 \in V(F),$ (ii) $\mathbb R^+(F)
\subseteq V(F)$ (ii) $-V(F)\cap  V(F)=\{0\}.$ We say that $V(F)$ is
$\vee$-{\it closed} if, for any chain $\{m_i\}$ from $V(F)$ bounded
in $\mathcal M^+(E),$ $\bigvee_i m_i \in V(F).$

If $F'$ is the complementary face of $F,$ according to Proposition
\ref{pr:3.2},  $m\in V(F')$ iff for any $t\in V(F)$ with $t\le^+m,$
we have $t=0.$

For an arbitrary cone $V$ of $\mathcal S(E),$ we denote by
$V^\sharp$ the set of all those measures $t\in \mathcal M^+(E)$ that
$m\le^+t$ for $m \in V$ entails $m=0.$ The elements of $V^\sharp$
are said to be $V$-{\it singular measures}. Hence, $V(F)^\sharp =
V(F')$ for any face $F$ of $\mathcal S(E)$ of a pseudo effect
algebra $E$ with (RDP).

\begin{theorem}\label{th:3.4}  Let $F$ be a face of the state space
$\mathcal S(E)$ of a pseudo effect algebra $E$ with {\rm (RDP)}.  If
$V(F)$ is $\vee$-closed, then $\mathcal S(E)$ equals the direct
convex sum of $F$ and its complementary face $F'.$

In addition, every measure $m$ on $E$ can be uniquely decomposed as
a sum
$$m=m_1 +m_1 \eqno(3.1)
$$
of two measures such that $m_1 \in V(F)$ and $m_2 \in V(F').$
\end{theorem}

\begin{proof}
{\it Existence:} Let $m$ be a measure on $E,$ and let $\Gamma(m)=\{m_1
\in V(F): m_1 \le^+ m\}.$ Since the zero measure belongs to
$\Gamma(m),$ $\Gamma(m)$ is non-void. Let $\{m_i\}$ be a chain of
elements from $\Gamma(m).$ The measure $m$  is an upper bound for
$\{m_i\}.$ By Theorem \ref{th:2.1}, $m_0:=\bigvee_i m_i$ is a measure on $E$
and by the hypotheses, $m_0$ belongs to $V(F).$ It follows from the
Zorn's Lemma that $\Gamma(m)$ contains a maximal element $m_1$ such
that $m_1\le^+ m.$

If we set $m_2 = m-m_1,$ we show that $m_2 \in V(F').$  Let $t \in
V(F)$ and let $t\le^+ m_2.$ Then $m_1+ t \le^+ m_1+m_2=m.$  Since
$m_1+t\in V(F),$ the maximality of $m_1$ in $\Gamma(m)$ implies
$t=0$,  and therefore, $m_2 \in V(F).$

\vspace{2mm}\noindent{\it Uniqueness:} Now let  $s$ be an arbitrary
state on $E.$  According to the first part of the present proof, we
see that $s$ can be decomposed in the convex form
$$s=\lambda_1 s_1 +\lambda_2s_2, \eqno(3.2)
$$
where $s_1 \in F$ and $s_2 \in F'.$   This implies that $\mathcal
S(E)$  is a direct convex sum of $F$ and $F'.$  In particular,
according to \cite[Prop 10.12]{Goo}, this yields that the
decomposition (3.2) is unique, i.e. if $s=\alpha_1 s_1'+\alpha_2s_2'
$ is another convex combination of $s_1'\in F$ and $s_2'\in F'$,
then $\alpha_i=\lambda_i$ for $i=1,2,$ and if $\alpha_i>0$ implies
$s_i = s_i'.$ In particular, this implies that  the decomposition in
(3.1) is unique.
\end{proof}

\section{Decomposition of States on Pseudo Effect Algebras}%4

This section presents the main results of the paper: Yosida--Hewitt
type and Lebesgue type of decomposition of finitely additive
measures and states.

Using closed faces $F$, we can decompose any state $s$ on a pseudo
effect algebra with (RDP) in the unique form: $s= \lambda s_1 +
(1-\lambda)s_2,$ where $s_1 \in F$ and $s_2 \in F',$ where $F'$ is
the complementary face of $F,$ see \cite[Thm 11.28]{Goo}. But not
every face of $\mathcal S(E)$ is closed. E.g. the set of all
$\sigma$-additive states on $E$ is a face that is not necessarily
closed. However, also for some such situations we show that
$\mathcal S(E)$ can be the direct convex sum of the face $F$ and its
complementary face.

A non-empty set $X$ of  a poset $E$ is {\it directed downwards}
({\it directed upwards}), and we write $D\downarrow$ ($D\uparrow$),
if for any $x,y \in X$ there exists $z \in D$ such that $z \le x,$
$z \le y$ ($z \ge x,$  $z \ge y$). Two downwards directed sets
$\{x_t:\ t \in T\}$ and $\{y_t:\ t \in T\}$ indexed by the same
index set $T$ are called {\it downwards equidirected} if, for any
$s,t \in T,$ there exists $v \in T$ such that $x_v \le x_s$ and $x_v
\le x_t$ as well as $y_v \le y_s$ and $y_v \le y_t.$ A similar
definition holds for upwards directed sets.

Let $x \in E$ and $D \subseteq E.$ We say that $D\uparrow x$ if
$D\uparrow$ and $x = \bigvee D.$ Dually we define $D\downarrow x,$
i.e. $D\downarrow$ and $x = \bigwedge D.$

Let $\{a_n\}$ be a sequence of elements of a pseudo effect $E$ such
that $b_n =a_1+\cdots +a_n$ exists in $E$ for each $n\ge 1$ and if
$a =\bigvee_n b_n$ exists in $E$, we write $a = \sum_n a_n.$

A signed measure $m$ on a pseudo effect algebra $E$ is $\sigma$-{\it
additive} if, $\{a_n\} \nearrow a,$ i.e. $a_n\le a_{n+1}$ for each
$n\ge 1$ and $\bigvee_n a_n = a,$ then $m(a)=\lim_n m(a_n).$ A
signed measure $m$ is $\sigma$-additive iff $\{a_n\}\searrow 0$
entails $\lim_nm(a_n)= 0.$

Now let $E$ be an effect algebra (not a pseudo effect algebra). We
say that a system $\{a_t\}_{t \in T}$ of elements of $E$ is {\it
summable} if, for each finite subset $F \subseteq T,$ the element
$a_F =\sum_{t\in F} a_t$ is defined in $E$. If  there exists the
element $a=\bigvee \{a_F: F \subseteq T\},$ we called it the {\it
sum} of the summable system $\{a_t\}_{t \in T},$ and write $a=
\sum_{t \in T} a_t.$

A signed measure $m$ on an effect algebra is said to be {\it
completely additive} if $m(a) = \sum_{t\in T}  m(a_t)$ whenever $a =
\sum_{t\in T} a_t.$

We denote by $\mathcal J(E)_\sigma, \mathcal J(E)_{\rm ca}$ the sets
of all $\sigma$-additive Jordan signed measures and completely
additive Jordan signed measures on $E,$ respectively. In the same
way we define $\mathcal M^+(E)_\sigma$, $\mathcal M^+(E)_{\rm ca}$
and similarly for the states: $\mathcal S(E)_\sigma$ and $\mathcal
S(E)_{\rm ca}$ denote the systems of all $\sigma$-additive and
completely additive states on $E.$

Then $\mathcal S(E)_{\rm ca}\subseteq \mathcal S(E)_\sigma \subseteq
\mathcal S(E).$  Each of these sets can be empty. Moreover,
$\mathcal S(E)_\sigma$ and $\mathcal S(E)_{\rm ca}$ are also faces
of $\mathcal S(E),$ and
$$V(\mathcal S(E)_\sigma) = \mathcal
M^+(E)_\sigma \quad  \mbox{and}  \quad V(\mathcal S(E)_{\rm ca}) =
\mathcal M^+(E)_{\rm ca}. \eqno(4.1)
$$

\begin{proposition}\label{pr:4.1}  Let $m_1,\ldots,m_n$ be
completely additive measures on an effect algebra $E$ that satisfies
{\rm (RDP)}. Then $m=m_1\vee\cdots \vee m_n$ is also a completely
additive measure on $E.$

The same is true if $m_1,\ldots,m_n$ are $\sigma$-additive measures
on a pseudo effect algebra $E$ with {\rm (RDP)}.

\end{proposition}

\begin{proof}  Without loss of generality, we can assume that $n=2.$
Let $a = \sum_{t\in T}a_t$ and let $F$ be any finite subset of $T$
and let $a_F = \sum_{t\in F} a_t.$ Given $\epsilon >0,$ there is
$F_0$ such that, for each finite $F \supseteq F_0,$ $m_1(a-a_F) <
\epsilon/2$ and $m_2(a-a_F)<\epsilon /2.$

By Theorem \ref{th:2.1}, $m(a-a_F) = \sup\{m_1(x)+m_2(y): x+y =
a-a_F\}.$  Since $x\le a-a_F$ and $y\le a-a_F,$ we have $m_1(x)\le
m_1(a-a_F) <\epsilon/2$ and $m_2(y) \le m_2(a-a_F)<\epsilon/2.$ Then
$m_1(x)+m_2(y) <\epsilon,$ consequently $m(a-a_F)\le \epsilon.$
Hence, $m(a)-m(a_F) \le \epsilon $ and $m$ is completely additive.
\end{proof}

\begin{lemma}\label{le:4.2}  If $\{f_t\}_{t\in T}\uparrow f$ and $\{g_t\}_{t\in
T}\uparrow g,$ where $f_t,g_t,f,g \in \mathbb R^+$ for all $t\in T.$
Then
$$ \bigvee_{s,t}(f_s+g_t)=f+g, \eqno(4.2)
$$

If, in addition, $\{f_t\}_{t\in T}$ and $\{g_t\}_{t\in T}$ are
upwards equidirected, then
$$\{f_t+g_t\}_{t\in T} \uparrow f+g.\eqno(4.3)
$$

\end{lemma}

\begin{proof} We have $f_s+g_t \le  f+g$ for all $s,t \in T.$  If
$f_s+g_t \le x$ for some $x \in \mathbb R,$ then $f_s \le x-g_t$ so
that $f\le x-g_t$ and hence, $g_t\le x-f$ so that $g\le x-f$ and
$f+g\le x$ that gives, (4.2).

Now assume $\{f_t\}$ and $\{g_t\}$ are upwards equidirected.
It is clear that $f_t+g_t \le f+g.$ The equidirectness  entails
that, for all indices $s_0,t_0\in T$, there exists an index $t$ such
that $f_{s_0}\le f_t$ and $f_{t_0} \le g_t.$ Therefore, for all
indices $s_0$ and $t_0$, $f_{s_0}+g_{t_0} \le f_t+g_t\le f+g$
which by (4.2) gives (4.3).
\end{proof}

\begin{lemma}\label{le:4.3}  Let $E$ be a pseudo effect algebra with
{\rm (RDP)}. If $\{m_i\}$ is a chain of measures in $\mathcal
M^+(E)$ that is bounded above, then for $m_0=\bigvee_i m_i$ we have

$$m_0(a)=\sup_i m_i(a), \eqno(4.4)$$
for each $a \in E.$

\end{lemma}

\begin{proof} We assert that if $d(x)=\bigvee_i m_i(x),$ $x\in E,$ then $d$ is
additive, i.e., $d(x+y)=d(x)+d(y)$ whenever $x+y$ is defined in $E.$
This follows from the fact that $\{m_i(x)\}\uparrow d(x)$ and
$\{m_i(y)\}\uparrow d(y),$ are upwards equidirected  because
$\{m_i\}$ is a chain, and $\{m_i(x)+m_i(y)\} =\{m_i(x+y)\}\uparrow
(d(x)+d(y))$ by (4.3).

Since, $d$ is a measure such that $m_i\le^+ d \le^+ m_0,$ we
conclude $d=m_0.$
\end{proof}

Now we present a Yosida-Hewitt type of decomposition for measures on
$E.$

\begin{theorem}\label{th:4.4}  Let $E$ be an effect algebra with
{\rm (RDP)}. Then every measure $m$ on $E$ can be uniquely expressed
in the form
$$
m= m_1 + m_2,\eqno(4.5)
$$
where $m_1 \in \mathcal M^+(E)_{\rm ca}$ and $m_2$ is a finitely
additive measure on $E$ such that if $t \in \mathcal M^+(E)_{\rm
ca},$ such  $t \le^+m_2,$ then $t=0.$

In particular, every state $s$ on $E$ can be uniquely expressed as a
convex combination
$$ s = \lambda_1s_1 +\lambda_2 s_2,\eqno(4.6)
$$
where $s_1$ is a completely additive state and $s_2$ is a finitely
additive state such that if $\alpha s' \le^+ s_2$ for some
completely additive state $s'$ on $E$ and for some constant $\alpha
\ge 0,$ then $\alpha =0.$
\end{theorem}

\begin{proof}  First we show  that $\mathcal  M^+(E)_{\rm ca}$ is a $\vee$-closed
cone in $ \mathcal M^+(E).$ Thus let $\{m_i\}$ be a chain of
completely additive measures on $E$ that is bounded above by a
finitely additive measure $m'.$ By Theorem \ref{th:2.1}, there is
$m_0=\bigvee_i m_i$ that is finitely additive and $m_0\le^+ m'.$

We assert that $m_0$ is completely additive. Let $ a = \sum _{t \in
T} a_t$ exist in $L.$ Then from the monotonicity of $m_0$ we have
$m_0 (a_F) \le m _0(a),$ where $a_F := \sum_{t \in F} a_t$ for any
$F$ finite subset of the index set $T$.

Assume that $m_0(a_F) \le x$ for some $x \in \mathbb R^+$ and  for
any finite $F.$ Then $m_i(a_F) \le x$ for any $i$ and any $F.$ The complete
additivity of $m_i$ entails that $m_i(a) \le x$ for any $i,$ so that
by (4.4),  $m_0(a) = \sup_i m_i(a) \le x.$ This gives
$$m_0(a) = \sup_{F } \sum_{t \in F} m_0(a_t),
$$
consequently $m_0 \in \mathcal M^+(E)_{\rm ca},$ and $\mathcal
M^+(E)_{\rm ca}$ is a $\vee$-closed cone of $\mathcal J(E).$

Now let $m$ be an arbitrary finitely additive measure on $E.$ Since
all the conditions of Theorem \ref{th:3.4} are satisfied, we obtain
the unique decomposition $m=m_1+m_2$ in question. Similarly we have
(4.6).
\end{proof}

\begin{theorem}\label{th:4.5}  Let $E$ be a pseudo effect algebra with
{\rm (RDP)}. Then every measure $m$ on $E$ can be uniquely expressed
in the form
$$
m= m_1 + m_2,\eqno(4.7)
$$
where $m_1 \in \mathcal M^+(E)_{\sigma}$ and $m_2$ is a finitely
additive measure on $E$ such that if $t \in \mathcal
M^+(E)_{\sigma},$ $t\ge 0$, such  $t \le^+m_2,$ then $t=0.$

In particular, every state $s$ on $E$ can be uniquely expressed as a
convex combination
$$ s = \lambda_1s_1 +\lambda_2 s_2,\eqno(4.8)
$$
where $s_1$ is a $\sigma$-additive state and $s_2$ is a finitely
additive state such that if $\alpha s' \le^+ s_2$ for some
completely additive state $s'$ on $E$ and for some constant $\alpha
\ge 0,$ then $\alpha =0.$
\end{theorem}

\begin{proof} It follows the same steps as the proof of Theorem
\ref{th:4.4}.
\end{proof}

\begin{remark}\label{re:4.6}
{\rm Theorems \ref{th:4.5} and \ref{th:4.5} have been proved in [11,
12,  9, 17]. They are analogues of the classical Yosida--Hewitt
decomposition from \cite{YoHe}. In \cite{DeMo}, the component $m_2$
from  Theorem \ref{th:4.4} is said to be a {\it weakly purely
additive measure} and that from Theorem \ref{th:4.5} a {\it purely
additive measure}.}
\end{remark}

Now we present another Yosida--Hewitt type decomposition for  an
analogue of complete additivity of measures  for pseudo effect
algebra.  We say that a measure $m$ on $E$ is {\it  upwards
continuous} if $\{a_t\}\uparrow a$ entails $\{m(a_t)\}\uparrow
m(a).$ A measure $m$ is  upwards continuous  iff $\{a_t\}\downarrow
0$ implies $\{ m(a_t)\}\downarrow 0.$

For example, if $E$ is an effect algebra, then $m$ is completely
additive whenever  $m$ is upwards continuous. Indeed, let $m$ be an
upwards continuous measure and let $a =\sum_{t\in T} a_t.$  Given
any finite subset $F$ of indices we define $a_F = \sum_{t\in F}
a_t.$ Then $\{a_F\}_F$ is upwards directed and $\{a_F\}\uparrow a,$
so that $m(a)=\sum_t m(a_t).$

%???Conversely, let $m$ be completely additive and let $\{a_t\}\uparrow a.$ Let $K(a)$ be the set of all differences $a_t - a_s,$ where $a_t \ge a_s$ and $s,t \in T.$ Then $K(a)$ is a summable system such that $a= \sum (a_t-a_s)$ so that $m(a)=$ ???

\begin{theorem}\label{th:tot} Let $E$ be a pseudo effect algebra with
{\rm (RDP)}. Then every measure $m$ on $E$ can be uniquely expressed
in the form
$$
m= m_1 + m_2,
$$
where $m_1$ is an  upwards continuous measure and $m_2$ is a
finitely additive measure on $E$ such that if $t$ is an upwards
continuous measure with  $t \le^+m_2,$ then $t=0.$

In particular, every state $s$ on $E$ can be uniquely expressed as a
convex combination
$$ s = \lambda_1s_1 +\lambda_2 s_2,
$$
where $s_1$ is an upwards continuous  state and $s_2$ is a finitely
additive state such that if $\alpha s' \le^+ s_2$ for some upwards
continuous state $s'$ on $E$ and for some constant $\alpha
\ge 0,$ then $\alpha =0.$
\end{theorem}

\begin{proof}
Let $\mathcal M ^+(E)_{\rm uc}$ and $\mathcal S(E)_{\rm uc}$ be the
sets of upwards continuous measures and states, respectively, on
$E.$ Then $\mathcal S(E)_{\rm uc}$ is a face and $V(\mathcal
S(E)_{\rm uc})= \mathcal M^+(E).$ We show that $\mathcal M
^+(E)_{\rm uc}$ is a $\vee$-closed cone.  Indeed, let $\{m_i\}$ be a
chain in $\mathcal M ^+(E)_{\rm uc}.$ Then $\sup_tm(a_t)\le m(a).$
By (4.4), for $m = \bigvee_i m_i,$ we have $m(a)=\sup_i m_i(a)$ for
each $a\in E.$ Assume $\{a_t\}\uparrow a.$  Given $\epsilon >0,$
there is $i$ such that $m_i(a)> m(a)-\epsilon/2.$ Since $m_i$ is
upwards continuous, there is $a_t$ such that $m_i(a_t)> m_i(a) -
\epsilon/2.$ Then $m_i(a_t) \ge m(a)-\epsilon$ and by (4.4), $m(a_t)
> m(a)-\epsilon.$ Therefore, $m(a) = \sup_t m(a_t).$

For the final desired result, we apply Theorem  \ref{th:3.4}.
\end{proof}

In what follows, we present two types of the Lebesgue decomposition.

Let $m_1$ and $m_2$ be measures on a pseudo effect algebra $E.$ We
say that  (i) $m_1$ is {\it absolutely continuous with} respect to
$m_2$, and we write $m_1 \ll m_2$ if $m_2(a)=0$ implies $m_1(a) = 0$
for $a \in E.$  (ii) $m_1$ is $m_2$-{\it continuous}, and we write
$m_1 \ll_\epsilon m_2$ provided given $\epsilon
>0,$ there is a $\delta >0$ such that $m_2(a) <\delta$ yields
$m_1(a)<\epsilon.$
%(iii) $m_1$ is $m_2$-{\it singular} if, $m_3\in
%\mathcal M^+(E),$ $m_3 \ll_\epsilon m_2$ and $m_3\le^+m_2$, then
%$m_3=0,$
(iii) $m_1\perp m_2$ if there is an element $a\in E$ such
that $m_2(a)=0=m_1(a^-).$

It is clear that $m_1 \ll_\epsilon m_2$ entails $m_1\ll m_2.$

\begin{theorem}\label{th:4.7}  Let $E$ be a pseudo effect algebra
with {\rm (RDP)}.  Let $t$ be a fixed measure on $E.$ Then every
measure $m$ on $E$ can be uniquely expressed in the form

$$ m = m_1 +m_2, \eqno(4.9)
$$
where $m_1$ and $m_2$ are finitely additive measures on $E$ such
that $m_1\ll_\epsilon t$ and if $m'$ is any measure such that
$m'\ll_\epsilon t$ and $m_1\le ^+ m_2,$ then $m'=0$.  Moreover,
$m_2\wedge t=0.$

In particular, every state $s$ on $E$ can be uniquely expressed  as
a convex combination of two states $s_1$ and $s_1$ on $E,$
$$
s = \lambda_1s_1 +\lambda_2 s_2,\eqno(4.10)
$$
where $s_1\ll_\epsilon t$ and  if $s'$ is any state such that
$s'\ll_\epsilon t$ and $\alpha s'\le ^+  s_2,$ then $\alpha =0.$
\end{theorem}

\begin{proof} If $t=0,$ the statement is trivial. Suppose that
$t(1)>0$  and let $F(t)=\{s \in \mathcal S(E): s\ll_\epsilon t\}.$
Then $t_0=t/t(1) \in F(t)$ and $F(t)$ is a non-empty face of
$\mathcal S(E).$

Let $C(t):=\{s \in \mathcal M^+(E): s \ll_\epsilon t\}.$  Then
$C(t)$ is a nonempty cone that is a subcone of $\mathcal M^+(E),$
and $C(t)=V(F(t)).$ We claim that $C(t)$ is $\vee$-closed.  Let
$\{m_i\}$ be a chain from $C(t)$ that is bounded above by $m' \in
\mathcal M^+(E).$ If we define $d(x)=\bigvee_i m_i(x),$ according to
Lemma \ref{le:4.3}, we have $d =m_0 :=\bigvee_i m_i$ and
$\{m_i(1)\}\uparrow d(1).$ Since $\mathcal J(E)$ is an Abelian
Dedekind complete $\ell$-group, we have $\{m_0-m_i\}\downarrow 0.$
Therefore, $\{(m_0-m_i)(1)\}=\{m_0(1) -m_i(1)\}\downarrow 0.$ Thus
given $\epsilon >0,$ there is an index $i_0$ such that, for each
$m_i$ with $m_i\ge^+m_{i_0},$ we have $m_0(1)-m_i(1)<\epsilon/2.$

Fix the index $i_0.$ Given $\epsilon >0,$ there is $\delta >0$ such
that $t(a)<\delta$ implies $m_{i_0}(a)<\epsilon/2.$

Calculate,

$$m_0(a)= m_{i_0}(a) +(m_0-m_{i_0})(a) \le m_{i_0}(a)+ (m_0-m_{i_0})(1)\le
\epsilon/2 + (m_0-m_{i_0})(1) < \epsilon
$$
that gives $m_0\ll_\epsilon t$ and $m_0 \in C(t).$

Applying  the general result from Theorem \ref{th:3.4}, we have the
existence and uniqueness of (4.9).

Now we show that $m_2\wedge t=0.$  Let $k$ be a measure on $E$ such
that $k\le^+m_2$ and $k\le^+t$. Then $k\ll_\epsilon t$, so that
$k=0$ and whence $m_2\wedge t=0.$
\end{proof}

\begin{theorem}\label{th:4.8}  Let $E$ be a pseudo effect algebra
with {\rm (RDP)}.  Let  $t$ be a fixed measure on $E.$ Then every
measure $m$ on $E$ can be uniquely expressed in the form

$$ m = m_1 +m_2, \eqno(4.11)
$$
where $m_1$ and $m_2$ are finitely additive measures on $E$ such
$m_1\ll t$ and  if $m'$ is any measure on $E$ such that $m'\ll t$
and $m'\le^+m_2,$ then $m'=0.$

\end{theorem}

\begin{proof}  Let us define $V(t)=\{m \in \mathcal M^+(E): m \ll t\}$
and $C(t):=\{s\in \mathcal S(E): s \ll t\}.$ Then $F(t)$ is a face
of $\mathcal S(E)$ and it generates $V(t)= V(F(t)).$ Assume that
$\{m_i\}$ is a chain from $V(t)$ that is bounded in $\mathcal J(E).$
If we set $d(x)=\bigvee_i m_i(x)$ for each $x\in E,$ then according
to (4.4), $d=m_0:=\bigvee_i m_i$. Therefore, if $t(a)=0,$ then
$m_i(a)=0$ for each $i$ so that $m(a)=d(a)=0$ and $m_0\in V(t).$

The existence and uniqueness of (4.11) follows from Theorem
\ref{th:3.4}.
\end{proof}

Finally, we show a  relation among two  types of continuity of
measures.

\begin{proposition}\label{pr:4.9} Let $s_1$ and $s_s$ be two
$\sigma$-additive measures on a $\sigma$-complete MV-algebra $E.$
Then $s_1 \ll s_2$ if and only if $s_1\ll_\epsilon s_2.$
\end{proposition}

\begin{proof} It is clear that $s_1 \ll_\epsilon s_2$ entails
$s_1\ll s_2.$ Now let us suppose  $s_1 \ll s_2$ and let (ad
absurdum) $s_1\not\ll_\epsilon s_2.$ Then there is an $\epsilon >0$
such that for each $n\ge 1$, there is an $a_n \in E$ such that
$s_2(a_n) < 1/2^n$ and $s_1(a_n)\ge \epsilon.$  Set $a =
\bigwedge_{n=1}^\infty \bigvee_{n=k}^\infty a_k.$ Then
$$s_2(a) \le s_2(a_n\vee a_{n+1}\vee
\cdots) \le \sum_{k=n}^\infty s_2(a_k) < 1/2^{n-1},
$$
so that $s_2(a) = 0.$

On the other hand, $s_1(a) = \lim_n s_1(a_n\vee a_{n+1}\vee \cdots)
\ge \limsup_n s_1(a_n) \ge \epsilon$ that contradicts $s_1 \ll s_2.$
\end{proof}

We say that a measure $t$ on $E$ is {\it Jauch-Piron} if
$t(a)=t(b)=0$ entails there is an element $c\in E$ such that $a,b
\le c$ and $m(c).$ Every measure on an MV-algebra is Jauch-Piron.

Let $E$ be a pseudo effect with (RDP).  According to \cite[Thm
3.2]{Dvu}, we say that an element $a \in E$ is said to be {\it
central} or {\it Boolean}, if $a \wedge a^- =0.$  Then also $a\wedge
a^\sim =0$ and $a^-=a^\sim.$ Moreover, for any $x \in E,$ $x\wedge
a$ is defined in $E,$ and
$$x = (x\wedge a)+(x\wedge a^-).
$$
%Similarly,
%$$ (\bigvee_n x_n)\wedge a = \bigvee_n(x_n \wedge a).
%$$

Let $C(E)$ be the set of all central elements of $E.$  Then it is a
Boolean algebra that is a subalgebra of $E.$ If, in addition, $E$ is
monotone $\sigma$-complete, then $C(E)$ is  a Boolean
$\sigma$-algebra, \cite[Thm 5.11]{Dvu}.  We recall that a pseudo
effect algebra $E$ is said to be {\it monotone $\sigma$-complete}
provided, for any sequence $\{a_n\}$ from $E$ such that $a_n\le
a_{n+1}$ for each $n\ge 1,$ $a=\bigvee_n a_n$ is defined in $E.$

If $a,b$ are central elements and $t$ is a measure, then $t(a)+t(b) = t(a\wedge
b)+t(a\vee b),$ so that if, in addition,  $t(a)=t(b)=0,$ then
$t(a\vee b)=0.$

Now we present another Lebesgue type of decomposition for measures.
For  two measures  $m$ and $t$ on $E$ we write $m\ll_C t$ provided
$t(a)=0$ for  $a \in C(E)$ implies $m(a) =0.$  It is a weaker form
of $m\ll t.$

\begin{theorem}\label{th:4.10}  Let $E$ be a monotone  $\sigma$-complete
pseudo effect algebra with {\rm (RDP)}  and let $t$ be a $\sigma$-additive
measure on $E$. Every $\sigma$-additive measure $m$ on $E$ can be
uniquely decomposed in the form

$$m=m_1 + m_2 \eqno(4.12)
$$
such that $m_1,m_2$ are  $\sigma$-additive measures, $m_1\ll_C t,$ and
$m_2\perp t.$

In particular, every $\sigma$-additive measure $s$ can be uniquely
decompose in the form
$$
s=\lambda s_1 +(1-\lambda)s_2,\eqno(4.13)
$$
where $s_1$ and $s_2$ are $\sigma$-additive measures on $E$ such
that $s_1\ll t$ and $s_2\perp t.$
\end{theorem}

\begin{proof}
{\it Existence:} Let $\Ker(t)_C=\{a\in C(E): t(a)=0\}.$ The zero
element $0$ is central, so that $0\in \Ker(t)_C.$ We order the
elements of $\Ker(t)_C$ by $a\preceq b$ iff $a,b\in \Ker(t)_C$ and
$m(a)\le m(b).$  Then $\preceq$ is a partial order and now let
$\{a_i\}$ be a chain of elements from $\Ker(t)_C$ with respect to
$\preceq,$ and let $\delta = \sup_i m(a_i).$ Then either there is an
upper bound $a$ of $\{a_i\}$ itself or there is a sequence $\{a_n\}$
in $\{a_i\}$ such that $m(a_n)< m(a_{n+1}) \nearrow \delta.$ Set $a
=\bigvee_n a_n;$ then  $a \in \Ker(t)_C$ and $m(a)= \lim_n
m(a_n)=\delta,$ so that $a$ is an upper bound in $\Ker(t)_C$ for the
chain $\{a_i\}.$  Applying the Zorn Lemma, $\Ker(t)_C$ contains a
maximal element, say $a_0.$

Set $m_1(x)=m(x\wedge a^-_0)$ and $m_2(x)=m(x\wedge a_0)$ for each
$x \in E.$  Since $C(E)$ is a Boolean $\sigma$-algebra, and $x =
(x\wedge a^-_0)+(x\wedge a_0)$ for each $x \in E,$ we see that $m_1$
and $m_2$ are $\sigma$-additive measures on $E$ such that $m=
m_1+m_2.$  Since $t(a_0)=0,$ then $m_2(a^-_0)=m(a^-_0\wedge a_0)=0$
so that $m_2\perp t.$ We assert that $m_1 \ll_C t.$ If not, there is
an element $a \in \Ker(t)_C$ such that $m_1(a)=m(a\wedge a^-_0)>0.$
Since $a_0 \le a_0 \vee (a\wedge a^-_0) \in \Ker(t)_C$ which
contradicts the maximality of $a_0.$ This proves that $m_1\ll_C t.$

{\it Uniqueness:}  Let $F(t)_\sigma$ be the set of all
$\sigma$-additive states $s$ on $E$ such that $s \ll_C t.$  Then
$F(t)_\sigma$ is a face and due to Theorems
\ref{th:4.4}--\ref{th:4.5}, if $\{m_i\}$ is a bounded chain from
$V(F(t)_\sigma),$ then $m_0=\bigvee_im_i$ is a $\sigma$-additive
measure, and in view of (4.4), $m_0\ll_C t.$ This gives that
$V(F(t)_\sigma)$ is a $\vee$-closed cone.  Now let $m'$ be an
arbitrary $\sigma$-additive measure from $V(F(t)_\sigma)$ such that
$m'\le^+m_2.$ Then $m'(a_0)=0$ while $t(a_0)$ and $m'(a^-_0)\le
m_1(a^-_0)=0$ so that $m'=0.$  Therefore, $m_1 \in V(F(t)'_\sigma)$
so that by Theorem \ref{th:3.4} we have that the decomposition
(4.12) is unique.
\end{proof}

\begin{corollary}\label{co:3.5}
Let $X$ be any subset of a pseudo effect algebra with {\rm (RDP)},
and let $F=\{s \in \mathcal S(E): X \subseteq \Ker(s)\}.$  Every
measure $m$ on $E$ can be uniquely decomposed in the form
$$m = m_1 + m_2,$$
where $m_1\in V(F)$ and $m_2$ is $V(F)$-singular.

In particular, every state $s$ on $E$ can be uniquely expressed as a
convex combination
$$ m=\lambda s_1 +(1-\lambda)s_2,$$
where $s_1 \in F$ and $s_2 \in F'.$
\end{corollary}

\begin{proof}
Since $V(F)=\{m \in \mathcal M^+(E): X \subseteq \Ker(m)\}$ is by
(4.4) $\vee$-closed, the statements follow from Theorem
\ref{th:3.4}.
\end{proof}

\end{document}